\documentclass[letterpaper,11pt,reqno]{amsart}

\usepackage[utf8]{inputenc}

\usepackage{etex}

\usepackage[margin=1.3in]{geometry}
\usepackage{amsmath}
\usepackage{amsthm}
\usepackage{amssymb}
\usepackage{mathtools}
\usepackage{stmaryrd}
\usepackage{mathdots}
\usepackage{framed}
\usepackage{xspace,xcolor}
\usepackage[breaklinks,colorlinks,citecolor=teal,linkcolor=teal,urlcolor=teal,pagebackref,hyperindex]{hyperref}
\usepackage[all,cmtip]{xy}
\usepackage[capitalize]{cleveref}
\usepackage[alphabetic]{amsrefs}
\usepackage{array}

\theoremstyle{plain}
\newtheorem{thm}{Theorem}[section]
\crefname{thm}{Theorem}{Theorems}
\newtheorem{lem}[thm]{Lemma}
\newtheorem{prop}[thm]{Proposition}
\newtheorem{cor}[thm]{Corollary}

\newtheorem{problem}[thm]{Problem}
\crefname{problem}{Problem}{Problems}

\theoremstyle{definition}
\newtheorem{defi}[thm]{Definition}

\newtheorem{eg}[thm]{Example}

\crefname{step}{Step}{Steps}

\theoremstyle{remark}
\newtheorem{rmk}[thm]{Remark}

\def\N{{\mathbb N}}

\def\O{\mathcal{O}}

\def\fa{\mathfrak{a}}
\def\fb{\mathfrak{b}}

\def\fm{\mathfrak{m}}

\def\a{\alpha}

\def\g{\gamma}

\def\f{\phi}
\def\ff{\psi}

\def\l{\lambda}
\def\n{\nu}
\def\m{\mu}

\def\p{\pi}

\def\.{\cdot}
\def\^{\widehat}
\def\~{\widetilde}

\def\ov{\overline}

\def\inj{\hookrightarrow}

\def\({\left(}
\def\){\right)}

\renewcommand{\and}{ \ \ \text{ and } \ \ }

\newcommand{\fall}{ \ \ \text{ for all } \ \ }

\def\red{\mathrm{red}}

\def\loc{\mathrm{loc}}

\DeclareMathOperator{\Spec} {Spec}
\DeclareMathOperator{\Spf} {Spf}

\DeclareMathOperator{\ord} {ord}

\DeclareMathOperator{\invlim} {\varprojlim}

\DeclareMathOperator{\Hom} {Hom}

\DeclareMathOperator{\Cont} {Cont}

\mathchardef\mhyphen="2D

\newcommand{\cjc}[2]{{#1}^{{#2}\mhyphen\mathrm{jc}}}
\newcommand{\mjc}[1]{{#1}^{m\mhyphen\mathrm{jc}}}

\newcommand{\ac}[1]{{#1}^{\mathrm{ac}}}

\newcommand{\cjsc}[2]{{#1}^{{#2}\mhyphen\mathrm{jsc}}}
\newcommand{\mjsc}[1]{{#1}^{m\mhyphen\mathrm{jsc}}}
\newcommand{\njsc}[1]{{#1}^{n\mhyphen\mathrm{jsc}}}
\newcommand{\asc}[1]{{#1}^{\mathrm{asc}}}
\newcommand{\jsc}[1]{{#1}^{\mathrm{jsc}}}

\newcommand{\Mod}[1]{\ (\text{mod}\ #1)}

\begin{document}

\title{Jet closures and the local isomorphism problem}


\author{Tommaso de Fernex}
\address{Department of Mathematics, University of Utah, Salt Lake City, UT 48112, USA}
\email{{\tt defernex@math.utah.edu}}

\author{Lawrence Ein}
\address{Department of Mathematics, University of Illinois at Chicago, Chicago, IL 60607, USA}
\email{{\tt ein@@math.uic.edu}}

\author{Shihoko Ishii}
\address{Tokyo Woman's Christian University, 167-8585 Tokyo, Suginami, Japan}
\email{shihoko@lab.twcu.ac.jp}


\thanks{The research of the first author was partially supported by NSF Grant DMS-1402907 and
NSF FRG Grant DMS-1265285.
The research of the second author was partially supported by NSF Grant DMS-1501085.
The research of the third author was partially supported by JSPS grant (C), 16K05089.}


\begin{abstract}
If a morphism of germs of schemes induces isomorphisms of all local jet schemes,
does it follow that the morphism is an isomorphism?
This problem is called the local isomorphism problem.
In this paper, we use jet schemes to introduce various closure operations among ideals
and relate them to the local isomorphism problem. 
This approach leads to a partial solution of the local isomorphism problem, 
which is shown to have a negative answer in general and a positive one
in several situations of geometric interest.
\end{abstract}

\maketitle

\section{Introduction}

In \cite{Nas95}, Nash related singularities of algebraic varieties to the geometry
of arc spaces. Further studies of the structure of arc spaces
were carried out by Denef and Loeser in connection to motivic integration \cite{DL99}, 
and their results were later applied to study singularities of pairs
by Musta\c t\u a \cite{Mus02}. Several applications in birational geometry have
been obtained since. 
In this paper, we investigate new connections of arc spaces in commutative ring theory. 

Let $X$ be a scheme defined over a field $k$. 
For a nonegative integer $m$, let $X_m$ denote the $m$-jet scheme of $X$;
when $m = \infty$, we have the arc space $X_\infty$, which we also call the $\infty$-jet scheme of $X$.  
Given a point $x \in X$, for every $m$ we denote by 
\[
X_m^x := \p_m^{-1}(x)
\]
the fiber over $x$ of the natural projection $\p_m \colon X_m \to X$.
We call $X_m^x$ the \emph{scheme of local $m$-jets} of $X$. 

Let now $\f \colon (Y,y) \to (X,x)$
be a morphism of germs of schemes over $k$.
Suppose that for every $m \ge 0$ (including $m = \infty$) 
the induced morphism $\f_m \colon Y_m \to X_m$ maps
the scheme of local $m$-jets $Y_m^y$ of $Y$ isomorphically to 
the scheme of local $m$-jets $X_m^x$ of $X$.
Does it follows that $\f$ is an isomorphism?
What if we assume that $\f$ is a closed immersion?

We will refer to these questions 
as the \emph{local isomorphism problem}
and the \emph{embedded local isomorphism problem}.
Focusing on the local ring $\O_{X,x}$, the latter
can be interpreted as asking whether there is a nonzero ideal in this ring
which defines a subscheme of $X$ with same local jet schemes at $x$. 
This questions leads us to discover a new aspect to
commutative ring theory by means of jet schemes. 

Given any local $k$-algebra $R$, we introduce and study various 
closure operations among the ideals of $R$ which we call
\emph{$m$-jet closures}. These are defined for every $m \ge 0$, including $m=\infty$ 
where the operation is also called \emph{arc closure}. 
For every $m$, the $m$-jet closure of an ideal $\fa \subset R$ is the largest
ideal defining a scheme in $\Spec R$ whose scheme of local $m$-jets
is equal to the one of the scheme defined by $\fa$. 
We also introduce a notion of \emph{$m$-jet support closure}, 
where we put an analogous condition on the reduced scheme of local $m$-jets. 
For ideals of regular local rings, 
the intersection of all $m$-jet support closures of an ideal, for $m$ finite, 
is the same as the integral closure, but the two differ in general, 
with the first one giving a tighter closure operation.

One of the properties we establish is that the intersection of 
all $m$-jet closures of an ideal, for $m$ finite, is equal to the
arc closure of the ideal.
Using this, we prove that the embedded local isomorphism problem 
has a positive answer for a germ $(X,x)$
if and only if the zero ideal of $\O_{X,x}$ is arc closed.

This result prompts us to investigate which local $k$-algebras 
have the property that their zero ideals are arc closed. 
We find an example of a non-Noetherian ring whose zero ideal is not arc closed,
and this implies that in general the above questions have a negative answer, 
even assuming that $\f$ is a closed embedding. 
In the positive direction, we prove that the embedded local isomorphism problem
has a positive answer for several classes of schemes, among which 
are all varieties and homogeneous schemes over a field.

\subsection*{Acknowledgements} 

The local isomorphism problem was posed by Herwig Hauser in discussion with the third author, when she visited him at University of Innsbruck in 2007. We are very grateful to Herwig Hauser for giving us the motivation to study jet closures. We also thank Kei-ichi Watanabe for useful discussions, Hiraku Kawanoue for pointing out a mistake in a preliminary version of the paper, and the referee for useful comments.

\section{The local isomorphism problem}
\label{s:loc-isom-problems}

We work over a field $k$. 
In this paper, $\N$ denotes the set of nonnegative integers.

Let $X$ be an arbitrary scheme over $k$. 
For every $m \in \N$, we define the \emph{$m$-jet scheme} of $X$ to be the scheme $X_m$ 
representing the functor from $k$-schemes to sets given by 
\[
Z \mapsto \Hom_k(Z \times \Spec k[t]/(t^{m+1}),X).
\]
A point of $X_m$ is called an \emph{$m$-jet} of $X$; 
it corresponds to a map $\Spec K[t](t^{m+1}) \to X$ where $K$ is the residue field of the point. 
We will also consider \emph{$A$-valued $m$-jets} for any $k$-algebra $A$; these
are maps $\Spec A[t](t^{m+1}) \to X$, and correspond to the $A$-valued points of $X_m$.  

Truncations $k[t]/(t^{p+1}) \to k[t]/(t^{m+1})$, defined for $p > m$, induce natural 
projections $X_p \to X_m$ which are affine morphisms. By taking
the inverse limit, we define the \emph{arc space} 
(or \emph{$\infty$-jet scheme}) $X_\infty = \invlim_m X_m$. 
This is the scheme representing the functor from $k$-schemes to sets given by 
\[
Z \mapsto \Hom_k(Z \hat\times \Spf k[[t]],X).
\]
A point of $X_\infty$ is called an \emph{arc} (or \emph{$\infty$-jet}) of $X$.
It can be equivalently viewed as a map $\Spf K[[t]] \to X$ 
or a map $\Spec K[[t]] \to X$, where $K$ is the residue field of the point. 
Just like for jets, we will also consider \emph{$A$-valued arcs} for any $k$-algebra $A$, which
are maps $\Spec A[[t]] \to X$. 
If $X$ is affine (or quasi-compact and quasi-separated, see \cite{Bha16}), 
then $A$-valued arcs correspond to $A$-valued points of $X_\infty$. 

We denote by $\ff_m \colon X_\infty \to X_m$, $\p_m \colon X_m \to X$,
and $\p = \p_\infty \colon X_\infty \to X$ the natural projections.
For ease of notation, it is often convenient to let $m$ range in $\N \cup \{\infty\}$
and denote $A[[t]]$, for a $k$-algebra $A$, also by $A[t]/(t^\infty)$.

For more on jet schemes and arc spaces, we refer to \cite{IK03,Voj07,EM09}.

Given a point $x \in X$,
for every $m \in \N \cup \{\infty\}$ we denote
\[
X_m^x := \p_m^{-1}(x).
\]

\begin{defi}
We call $X_m^x$ the \emph{scheme of local $m$-jets} of $X$ at $x$.
\end{defi}

Roughly speaking, we would like to determine
how much of the local structure of the scheme $X$ at the point $x$ is encoded in the 
schemes of local jets $X_m^x$.

A related question asks how much of the geometry of a scheme $X$ is encoded in its 
jet schemes $X_m$ for $m \ge 1$ (including $m = \infty$). This question was addressed in \cite{IW10}, 
where it is shown that
the answer is both positive and negative depending on how one interprets the question. 
On the one hand, an example is given
of two non-isomorphic schemes $X$ and $Y$
whose jet schemes $X_m$ and $Y_m$ are isomorphic for all $m \ge 1$ in a compatible way with respect to the 
truncation morphisms $X_{m+1} \to X_m$ and $Y_{m+1} \to Y_m$. 
On the other hand, if the isomorphisms
$Y_m \to X_m$ are induced by a given morphism $\f \colon Y \to X$, then it is
proved that $\f$ must be an isomorphism. 

In view of the negative example in \cite{IW10}, we consider the following setting. 
Let
\[
\f \colon (Y,y) \to (X,x)
\]
is a $k$-morphism of germs of $k$-schemes. For every $m \in \N \cup\{\infty\}$, there is
an induced morphism of scheme of local $m$-jets
\[
\f_m^\loc \colon Y_m^y \to X_m^x
\]
given by restriction of the natural map $\f_m \colon Y_m \to X_m$. 

\begin{problem}[Local Isomorphism Problem]
\label{th:loc-isom-problem}
With the above notation, if $\f_m^\loc$ is an isomorphism for every $m \in \N \cup\{\infty\}$, 
does it follow that $\f$ is an isomorphism of germs?
\end{problem}

\begin{problem}[Embedded Local Isomorphism Problem]
\label{th:emb-loc-isom-problem}
With the above notation, assume that $\f$ is a closed immersion.
If $\f_m^\loc$ is an isomorphism for every $m \in \N \cup\{\infty\}$, 
does it follow that $\f$ is an isomorphism of germs?
\end{problem}

\begin{rmk}
Taking $m=0$, we see that the assumptions in either \cref{th:loc-isom-problem}
or \cref{th:emb-loc-isom-problem} include the condition that $\f$ induces 
an isomorphism between the residue fields $k(x)$ of $x$ and $k(y)$ of $y$. 
We will identity these fields, and denote them by $K$. 
Note that, for all $m$, the maps $\f_m^\loc \colon Y_m(y) \to X_m(x)$
are naturally $K$-morphisms (and hence $K$-isomorphisms under those assumptions).
\end{rmk}

\begin{defi}
We say that a germ $(X,x)$ has the \emph{embedded local isomorphism property}
if \cref{th:emb-loc-isom-problem} has a positive answer for every closed 
immersion $\f \colon (Y,y) \to (X,x)$. 
We also say that a germ $(X,x)$ has the \emph{local isomorphism property}
for a certain class of schemes
if \cref{th:loc-isom-problem} has a positive answer for every morphism
$\f \colon (Y,y) \to (X,x)$ with $(Y,y)$ in that class. 
\end{defi}

\begin{prop}
\label{th:emb-loc-isom-Noeth-implies-nonembedded}
Assume that $(X,x)$ has the embedded local isomorphism property.
Then $(X,x)$ has the local isomorphism property for all locally Noetherian $k$-scheme germs $(Y,y)$. 
\end{prop}

The proposition is an immediate consequence of the following property.

\begin{lem}
Let $\f \colon (Y,y) \to (X,x)$ be a $k$-morphism of germs inducing an isomorphism 
of the residue fields at $x$ and $y$ (which we identify and denote by $K$).
Assume that $Y$ is locally Noetherian at $y$ and 
$\f_1^\loc \colon Y_1^y \to X_1^x$ is a $K$-isomorphism. 
Then $\f$ is a closed immersion of germs.
\end{lem}

\begin{proof}
By hypothesis, the differential of $\f$ induces an isomorphism of $K$-vector spaces
between the Zariski tangent spaces $\fm_{Y,y}/\fm_{Y,y}^2$ of $Y$ at $y$
and $\fm_{X,x}/\fm_{X,x}^2$ of $X$ at $x$. This implies that
the homomorphism $\f^* \colon \O_{X,x} \to \O_{Y,y}$ corresponding to $\f$
induces a a surjection $\fm_{X,x} \to \fm_{Y,y}$. 
Indeed, we have $\fm_{Y,y} = \f^*(\fm_{X,x}) + \fm_{Y,y}^2$ by the isomorphism 
of the Zariski tangent spaces, and therefore we have 
$\fm_{Y,y}^i = \f^*(\fm_{X,x}^i) + \fm_{Y,y}^{i+1}$ for $i \ge 1$. 
By successive substitutions, we obtain
$\fm_{Y,y} = \f^*(\fm_{X,x}) + \fm_{Y,y}^n$ for all $n \ge 1$.
Since $Y$ is Noetherian at $y$, this implies that $\fm_{Y,y} = \f^*(\fm_{X,x})$.
By this and the isomorphism between the residue fields, we conclude that
$\f^* \colon \O_{X,x} \to \O_{Y,y}$ is surjective.
\end{proof}

When restricting to the embedded setting, we often denote the germ by $(X,o)$. 
Later in \cref{s:loc-isom-problems-revisited},
we shall reinterpret the embedded local isomorphism problem
from an algebraic point of view, relating it to certain ``jet theoretic''
notions of closure of ideals in local rings.

\section{Jet closures}
\label{s:jet-closures}

Throughout this section, let $(R,\fm)$ be a local algebra over a field $k$, 
and let $X = \Spec R$. We denote by $o \in X$ the closed point. 

For every $m \in \N \cup \{\infty\}$, let $R_m$ be the ring
of regular functions of the $m$-jet scheme $X_m$, so that 
$X_m = \Spec R_m$. For an ideal $\fa \subset R$ and $m \in \N \cup \{\infty\}$, we denote
\[
\fa_m := (D_i(f) \mid f \in \fa,\, 0 \le i < m+1 ) \subset R_m
\]
the ideal generated by the Hasse--Schmidt derivations $D_i(f)$
of the elements $f$ of $I$. If $V(\fa) \subset X$ is the subscheme defined by $\fa$, 
then its $m$-jet scheme $V(\fa)_m$ is the subscheme of $X_m$ defined by $\fa_m$. 

\begin{defi}
For any ideal $\fa \subset R$ and any $m \in \N \cup \{\infty\}$, we define
the \emph{$m$-jet closure} of $\fa$ to be 
\[
\mjc\fa := \{ f \in R \mid (f)_m \subset \fa_m \Mod{\fm R_m} \}.
\]
For $m=\infty$, we also call the ideal 
\[
\ac\fa := \cjc\fa\infty
\]
the \emph{arc closure} of $\fa$. 
\end{defi}

Unless otherwise stated,
in the following statements we let $m$ be an arbitrary element of $\N \cup \{\infty\}$
We start with the following property which implies that 
the $m$-jet closure an ideal $\fa$ is intrinsic, that is, 
only depends on the quotient ring $R/\fa$ and hence on the scheme $\Spec R/\fa$, 
and not by the embedding of $\Spec R/\fa$ in $\Spec R$. 

\begin{lem}
\label{th:intrinsic}
The $m$-jet closure $\mjc\fa$ of an ideal $\fa \subset R$
is the inverse image via the quotient map $R \to R/\fa$ of the 
$m$-jet closure of the zero ideal of $R/\fa$.
\end{lem}

\begin{proof}
Let $f \in R$ be an element and let 
$\ov f$ denote the class of $f$ in $R/\fa$.
Note that the image of $(f)_m$ in $R_m/\fa_m$ is equal to the ideal $(\ov f)_m$.
The condition that $(f)_m \subset \fa_m$
modulo $\fm R_m$ is equivalent to the condition that $(\ov f)_m =0$ modulo $\fm R_m/\fa_m$. 
It follows that $f \in \mjc\fa$ if and only if $\ov f$ is in the $m$-jet closure
of the zero ideal of $R/\fa$. 
\end{proof}

In view of this fact, in order to study properties
of $m$-jet closures we can reduce to the case of the zero ideal $(0) \subset R$, 
after replacing $R$ by $R/\fa$. 
In this special case, the $m$-jet closure can be expressed in terms of the universal $m$-jet
homomorphism. 

Recall that the universal $m$-jet
\[
\m_m \colon X_m \hat\times \Spf k[t]/(t^{m+1}) \to X
\] 
is defined by the ring homomorphism
\[
\m_m^* \colon  R \to R_m[t]/(t^{m+1}), \quad f \mapsto \sum_{i=0}^m D_i(f)t^i.
\]
We denote by $\l_m$ the restriction of $\m_m$ to the scheme of local $m$-jets $X_m^o$, 
and consider the corresponding ring homomorphism
\[
\l_m^* \colon R \to (R_m/\fm R_m)[t]/(t^{m+1}).
\]

\begin{lem}
\label{th:mjc=ker}
We have
\[
\mjc{(0)} = \ker \l_m^*.
\]
\end{lem}

\begin{proof}
It suffices to observe that, by definition,
\[
\mjc{(0)} = \{ f \in R \mid (f)_m = 0 \Mod{\fm R_m} \} = \ker \l_m^*.
\]
\end{proof}

The next property provides a geometric characterization of $m$-jet closures. 

\begin{prop}
\label{th:geom-char}
The $m$-jet closure $\mjc\fa$ of an ideal $\fa \subset R$
is the largest ideal $\fb \subset R$ such that $V(\fb)_m^o = V(\fa)_m^o$. 
\end{prop}

\begin{proof}
By \cref{th:intrinsic}, it suffices to prove the proposition for the zero ideal $(0)$
of $R$. The fact that $\mjc{(0)}$ is an ideal is clear by \cref{th:mjc=ker}. 
The second assertion follows by the geometric reinterpretation of
the definition of $m$-jet closure:
\[
\mjc\fa = \{ f \in R \mid V(f)_m^o \supset V(\fa)_m^o \}.
\]
In the case of the zero ideal, we have 
\[
\mjc{(0)} = \{ f \in R \mid V(f)_m^o = X_m^o \},
\]
and this implies the assertion. 
\end{proof}

\begin{rmk}
Now that we know that the $m$-jet closure of an ideal is itself an ideal, we can 
rephrase their intrinsic property by saying that,
for every ideal $\fa \subset R$, there is a canonical isomorphism
of $k$-algebras
\[
\frac{R}{\mjc\fa} \cong \frac{R/\fa}{\mjc{((0)R/\fa)}}.
\]
\end{rmk}

The next corollary implies that the $m$-jet closure
is indeed a closure operation.

\begin{cor}
\label{th:closure-operation}
For any ideal $\fa \subset R$ we have 
\[
\fa \subset \mjc\fa = \mjc{(\mjc\fa)}.
\]
\end{cor}

\begin{proof}
Both the inclusion $\fa \subset \mjc\fa$ and the equality $\mjc\fa = \mjc{(\mjc\fa)}$
are immediate consequences of \cref{th:geom-char}.
\end{proof}

\begin{defi}
We say that an ideal $\fa \subset R$ is \emph{$m$-jet closed}
if $\fa = \mjc\fa$. When $m = \infty$ (where the condition can be written
as $\fa = \ac\fa$), we also say that $\fa$ is \emph{arc closed}.
\end{defi}

\begin{cor}
For two ideals $\fa,\fb \subset R$
we have $\mjc\fa = \mjc\fb$ if and only if 
$V(\fa)_m^o = V(\fb)_m^o$. 
\end{cor}

\begin{proof}
This is also an immediate consequences of \cref{th:geom-char}.
\end{proof}

\begin{prop}
\label{th:m-adic}
For any ideal $\fa \subset R$ and any $m \in \N$, we have $\fa + \fm^{m+1} \subset \mjc\fa$.
\end{prop}

\begin{proof}
By \cref{th:intrinsic}, it suffices to prove the case where $\fa = (0) \subset R$.
Recall that $\mjc{(0)} = \ker\l_m^*$ by \cref{th:mjc=ker}. 
Since $\l_m^*(\fm) \subset (t)$, we have
$\l_m^*(\fm^{m+1}) \subset (t^{m+1}) = 0$, and therefore $\fm^{m+1} \subset \ker\l_m^* = \mjc{(0)}$.
\end{proof}

\begin{rmk}
\label{r:non-mjc}
\cref{th:m-adic} implies in particular that even in nice situations (e.g., assuming that $R$ is
a Noetherian $k$-algebra) the $m$-jet closure operation on ideals is, for finite $m$, 
a nontrivial operation. For instance, if $\fa$ is not $\fm$-primary, then 
$\fa \subsetneq \mjc\fa$ for all $m \in \N$.
\end{rmk}

The following is an example where $\fa + \fm^{m+1} \ne \mjc\fa$.

\begin{eg}
Let $\fa = (x^2+y^3) \subset R = k[[x,y]]$, where $k$ is a field of characteristic $\ne 2,3$.
A direct computation shows that $\cjc\fa 4 = (x^2 + y^3,xy^3)$, and hence 
$\fa + \fm^5 \ne \cjc\fa 4$. To see this, we introduce the 
coordinates $x_i := D_i(x)$ and $y_i := D_i(y)$ on the jet schemes of $X = \Spec R$, 
where $(D_i)_{i \ge 0}$ is the sequence of universal Hasse--Schmidt derivations. 
Denoting by $o \in X$ the closed point, we have
$X_4^o = \Spec k[x_i,y_i \mid 1 \le i \le 4]$, and the ideal 
defining $V(\fa)_4^o$ in there is generated by the three elements
\[
x_1^2,\quad 
2x_2x_1 + y_1^2,\quad 
2x_3x_1 + x_2^2 + 3y_2y_1^2.
\]
One can check that $x_1y_1^3$ is in this ideal, and this implies that
$xy^3 \in \cjc\fa 4$. On the other hand, $xy^3 \not\in \fa + \fm^5$. 
\end{eg}

Recall that we defined the arc closure to be $\ac\fa := \cjc\fa\infty$. The following proposition
shows how the arc closure compares to the other $m$-jet closures. 

\begin{prop}
\label{th:cap-mjc=ac}
For any ideal $\fa \subset R$, we have 
\[
\bigcap_{m \in \N} \mjc\fa = \ac\fa.
\]
\end{prop}

\begin{proof}
By \cref{th:intrinsic}, we reduce to prove the formula when $\fa$ is the zero ideal of $R$, 
where by \cref{th:mjc=ker} the proposition is equivalent to the assertion that 
\[
\bigcap_{m \in \N} 
\ker(\l_m^*) = \ker(\l_\infty^*).
\]

To show one inclusion, let $f \in R$ be an element such that 
$\l_\infty^*(f) \ne 0$ in $(R_\infty/\fm R_\infty)[[t]]$. 
Then the image
$\l_\infty^*(f)$ in $(R_\infty/\fm R_\infty)[t]/(t^{m+1})$ is nonzero for all sufficiently large 
integers $m$.
Since, for any such $m$, the resulting map $R \to (R_\infty/\fm R_\infty)[t]/(t^{m+1})$
factors through $(R_m/\fm R_m)[t]/(t^{m+1})$ by the universality of $\l_m^*$, 
it follows that $\l_m^*(f) \ne 0$, and hence we have 
$f \not\in\bigcap_{m \in \N}\ker(\l_m^*)$.

For the reverse inclusion, assume that $f \not\in \ker(\l_m^*)$ for some $m \in \N$. 
Then there exists an $A$-valued $m$-jet $\g \colon \Spec A[s]/(s^{m+1}) \to \Spec R$, 
for some $k$-algebra $A$, such that $f$ is not in the kernel of the homomorphism
\[
\g^* \colon R \to A[s]/(s^{m+1}).
\]
By further composing with the homomorphism
\[
A[s]/(s^{m+1}) \to \big(A[s]/(s^{m+1})\big)[[t]], \quad s \mapsto st,
\]
which is injective, we obtain an $A[s]/(s^{m+1})$-valued 
arc $\alpha \colon \Spec\big(A[s]/(s^{m+1})\big)[[t]] \to \Spec R$,
and $f$ is not in the kernel of 
\[
\a^* \colon R \to \big(A[s]/(s^{m+1})\big)[[t]].
\]
This implies that $f \not\in \ker(\l_\infty^*)$.
\end{proof}

\section{Jet support closures}

As in the previous section, let $(R,\fm)$ be a local algebra over a field $k$, let
$X = \Spec R$, and let $o \in X$ be the closed point.   
By looking at the local jet schemes of $(X,o)$ with their reduced structure, 
we introduce the following variant of the $m$-jet closure. 

\begin{defi}
For any ideal $\fa \subset R$ and any $m \in \N \cup \{\infty\}$, we define
the \emph{$m$-jet support closure} of $\fa$ to be 
\[
\mjsc\fa := \{ f \in R \mid (f)_m \subset \sqrt{\fa_m} \Mod{\fm R_m} \}.
\]
For $m=\infty$, we also call the ideal 
\[
\asc\fa := \cjsc\fa\infty
\]
the \emph{arc support closure} of $\fa$. 
\end{defi}

The $m$-jet support closure operation satisfies many analogous properties 
of the $m$-jet closure operation studied in the previous section that can 
be proven in a similar way. 

For instance, 
the $m$-jet support closure is intrinsic. That is, 
for any ideal $\fa \subset R$, the $m$-jet support closure $\mjc\fa$ 
is the inverse images via the quotient map $R \to R/\fa$ of the 
$m$-jet support closure of the zero ideal of $R/\fa$.

Furthermore, the $m$-jet support closure of the zero ideal of $R$ is given by 
\[
\mjsc{(0)} = \ker \n_m^*,
\]
where
\[
\n_m^* \colon R \to (R_m/\fm R_m)_\red[t]/(t^{m+1})
\]
is the ring homomorphism associated with the
restriction $\n_m$ of the universal $m$-jet morphism $\m_m$ 
to the reduced scheme of local $m$-jets $(X_m^o)_\red$. 

It follows from these two properties that for any ideal $\fa \subset R$ the 
$m$-jet support closure $\mjsc\fa$ is an ideal. In general, we have
\[
\fa \subset \mjsc\fa = \mjsc{(\mjsc\fa)},
\]
which says that the $m$-jet support closure is a closure operation among the ideals of $R$.

\begin{defi}
We say that an ideal $\fa \subset R$ is \emph{$m$-jet support closed}
if $\fa = \mjsc\fa$. When $m = \infty$ (where the condition can be written
as $\fa = \asc\fa$), we also say that $\fa$ is \emph{arc support closed}.
\end{defi}

The following comparison between $m$-jet support closure and $m$-jet closure is clear from the definitions. 

\begin{prop}
\label{th:mjc<mjsc}
For every ideal $\fa \subset R$, 
we have $\mjc\fa \subset \mjsc\fa$. 
\end{prop}

Geometrically, we have
\[
\mjsc\fa = \{ f \in R \mid (V(f)_m^o)_\red \supset (V(\fa)_m^o)_\red \}.
\]
In the case of the zero ideal, we have 
\[
\mjsc{(0)} = \{ f \in R \mid (V(f)_m^o)_\red = (X_m^o)_\red \},
\]
Using these facts, the same argument as in the proof of \cref{th:geom-char}
gives the following property. 

\begin{prop}
\label{th:geom-char-mjsc}
The $m$-jet support closure $\mjsc\fa$ of an ideal $\fa \subset R$
is the largest ideal $\fb \subset R$ such that $(V(\fb)_m^o)_\red = (V(\fa)_m^o)_\red$.
\end{prop}

\begin{cor}
\label{th:a-b-same-mjc}
For two ideals $\fa,\fb \subset R$,
we have $\mjsc\fa = \mjsc\fb$ if and only if 
$(V(\fa)_m^o)_\red = (V(\fb)_m^o)_\red$. 
\end{cor}

There is one proof from the previous section that does not have an analogue
for $m$-jet support closure, and that is the proof of \cref{th:cap-mjc=ac}. 
For this reason, it is convenient to give the following definition. 

\begin{defi}
The \emph{jet support closure} of an ideal $\fa \subset R$ is the ideal
\[
\jsc\fa := \bigcap_{m \in \N} \mjsc\fa.
\]
\end{defi}

The fact that this is a closure operation is proven in the next proposition. 

\begin{prop}
\label{th:jsc-closure-operation}
For every ideal $\fa \subset R$, we have 
\[
\fa \subset \jsc\fa = \jsc{(\jsc\fa)}.
\]
\end{prop}

\begin{proof}
The inclusion $\fa \subset \jsc\fa$ is clear from the definition. 
Regarding the equality, we have
\[
\jsc{(\jsc\fa)} 
= \bigcap_{m \in \N} \mjsc{\(\bigcap_{n \in \N} \njsc\fa\)}
\subset \bigcap_{m \in \N} \mjsc{(\mjsc\fa)}
= \bigcap_{m \in \N} \mjsc\fa
= \jsc\fa, 
\]
and since $\jsc\fa \subset \jsc{(\jsc\fa)}$, the two ideals are the same. 
\end{proof}

\begin{defi}
We say that an ideal $\fa \subset R$ is \emph{jet support closed} if $\fa = \jsc\fa$.
\end{defi}

Since all $m$-jet support closures are intrinsic, it follows that so is the jet support closure. 
In particular, we have the following property. 

\begin{prop}
\label{th:jsc-intrinsic}
An ideal $\fa \subset R$ is jet support closed if and only if the zero ideal of $R/\fa$ is
jet support closed. 
\end{prop}

An adaptation of the first part of the proof of \cref{th:cap-mjc=ac} gives
the following comparison between jet support closure and arc support closure. 
We do not know whether the reverse inclusion holds.  

\begin{prop}
For any ideal $\fa \subset R$, we have $\jsc\fa \subset \asc\fa$.
\end{prop}

The jet support closure compares to the arc closure as follows. 

\begin{prop}
\label{th:ac<jcs}
For any ideal $\fa \subset R$, we have $\ac\fa \subset \jsc\fa$.
\end{prop}

\begin{proof}
By \cref{th:mjc<mjsc}, we have  
\[
\bigcap_{m \in \N} \mjc\fa \subset \bigcap_{m \in \N} \mjsc\fa = \jsc\fa,
\]
and hence the assertion follows from \cref{th:cap-mjc=ac}.
\end{proof}

\begin{prop}
\label{jsc<integral-closure}
Let $R$ be a local integral domain essentially of finite type over a field $k$. 
Let $\fa \subset R$ be an ideal, and
denote by $\ov\fa$ its integral closure. 
\begin{enumerate}
\item
There is an inclusion $\jsc\fa \subset \ov\fa$. 
\item
If $R$ is regular, then $\jsc\fa = \ov\fa$. 
\end{enumerate}
\end{prop}

\begin{proof}
First, we claim that 
\begin{equation}
\label{eq:ord_a}
\ord_\a(\jsc\fa) = \ord_\a(\fa)
\end{equation}
for every arc $\a \in X_\infty^o$. Suppose this is not the case for some $\a$. 
Since $\fa \subset \jsc\fa$, we must have $\ord_\a(\jsc\fa) < \ord_\a(\fa)$. 
Setting $m = \ord_\a(\jsc\fa)$, this means that there is an element $f \in \jsc\fa$
such that $\ord_\a(f) = m < \ord_\a(\fa)$. This implies that 
$\ff_m(\a) \in (V(\fa)_{m-1}^o)_\red \setminus (V(f)_{m-1}^o)_\red$. 
The contradiction then follows by \cref{th:geom-char-mjsc},
after we observe that $f \in \mjsc\fa$. 
This proves the claim.

Every divisorial valuation $v$ on $X = \Spec R$ determines an arc 
$\a_v \colon \Spec k_v[[t]] \to X$ given by
\[
\a_v^* \colon R \to \O_v \to \^\O_v \xrightarrow{\simeq} k_v[[t]].
\]
Here $\O_v$ is the valuation ring of $v$, $k_v = \O_v/\fm_v$ is the residue field,
$\^\O_v$ is the $\fm_v$-adic completion, 
and the isomorphism $\^\O_v \simeq k_v[[t]]$ is given by Cohen Structure Theorem. 
Since for every divisorial valuation $v$ we have $\ord_{\a_v} = v$, 
\cref{eq:ord_a} implies that $v(\jsc\fa) = v(\fa)$ for every divisorial valuation $v$.
It follows that $\ov{\jsc\fa} = \ov\fa$
by the valuative characterization of integral closure, and hence we have $\jsc\fa \subset \ov\fa$.
This proves~(a). 

Assume now that $R$ is regular. For any ideal $\fb \subset R$ and any $m \in \N$, denote
\[
\Cont^{> m}(\fb) := \{ \a \in X_\infty \mid \ord_\a(\fb) > m \}.
\]
Note that
\[
\Cont^{> m}(\fb) = \big(\ff_m^{-1}((V(\fb)_m)_\red)\big)_\red.
\]
We have 
\[
\Cont^{> m}(\fa) = \Cont^{> m}(\ov\fa) \fall m \in \N.
\]
Since $X$ is smooth, the projections $\ff_m \colon X_\infty \to X_m$ are surjective, 
and therefore we have $(V(\fa)_m)_\red = (V(\ov\fa)_m)_\red$ for all $m \in \N$. 
By restricting to the reduced fibers over $o \in X$, we see that
\[
(V(\fa)_m^o)_\red = (V(\ov\fa)_m^o)_\red  \fall m \in \N,
\]
and therefore we have $\jsc\fa = \jsc{(\ov\fa)}$ by \cref{th:a-b-same-mjc}
and the definition of jet support closure. 
On the other hand, part~(a) applied to $\ov\fa$ implies that $\jsc{(\ov\fa)} = \ov\fa$. 
Therefore we have $\jsc\fa = \ov\fa$, which proves~(b).
\end{proof}

The following example shows that in general, if $R$ is not regular, the
jet support closure is a tighter operation than integral closure. 

\begin{eg}
Let $R = k[[x,y,z]]/(x^2+y^2+z^2)$ and $\fa = (x,y) \subset R$
(with slight abuse of notation, we still denote by $x,y,z$ the classes
of these elements in $R$). Since $z$ is integral over $\fa$, we have 
$\ov\fa = (x,y,z)$. On the other hand, note that
$R/\fa \cong k[[x]]/(x^2)$. 
Since $(x^2)$ is integrally closed in the ring $k[[x]]$, 
it is also jet support closed $k[[x]]$.
It follows by applying \ref{th:jsc-intrinsic} twice that
$\fa$ is jet support closed in $R$. 
\end{eg}

\section{The embedded local isomorphism problem revisited}
\label{s:loc-isom-problems-revisited}

Our first result of this section provides a characterization of the embedded local isomorphism problem
in terms of jet closures.

\begin{prop}
\label{th:emb-loc-iso-revised}
Let $(R,\fm)$ be a local $k$-algebra, 
and let $X = \Spec R$ with closed point $o \in X$. 
Then the germ $(X,o)$ has the embedded local isomorphism property
if and only if the zero ideal of $R$ is arc closed.
\end{prop}

\begin{proof}
If $\f \colon (Y,o) \inj (X,o)$ is closed immersion of $k$-scheme germs
and $I_Y \subset R$ is the ideal of $Y$, then for any $m \in \N \cup \{\infty\}$ the 
map $\f_m^\loc \colon Y_m^o \to X_m^o$ is an isomorphism if and only if 
$I_Y \subset \mjc{(0)}$. 
This implies that 
$(X,o)$ has the embedded local isomorphism property if and only if
\[
\bigcap_{m \in \N \cup \{\infty\}}\mjc{(0)} = (0).
\]
By \cref{th:cap-mjc=ac}, this last condition is equivalent to
the condition that $\ac{(0)} = (0)$.
\end{proof}

\begin{rmk}
What \cref{th:emb-loc-iso-revised} says is that \cref{th:emb-loc-isom-problem}
would not change if the condition that $\f_m^\loc \colon Y_m^o \to X_m^o$ 
is an isomorphism for all $m \in \N \cup \{\infty\}$ is replaced with the 
weaker requirement that just $\f_\infty^\loc \colon Y_\infty^o \to X_\infty^o$ is an isomorphism.
\end{rmk}

\begin{rmk}
In view of \cref{th:cap-mjc=ac}, \cref{th:emb-loc-iso-revised} can equivalently be stated 
by saying that $(X,o)$ has the embedded local isomorphism property if and only if
\[
\bigcap_{m \in \N}\mjc{(0)} = (0).
\]
In the formulation of \cref{th:emb-loc-isom-problem}, this means that
it is equivalent to only assume that $\f_m^\loc \colon Y_m^o \to X_m^o$ 
is an isomorphism for all $m \in \N$, 
instead of assuming it for all $m \in \N \cup \{\infty\}$.
\end{rmk}

In view of \cref{th:emb-loc-iso-revised}, it is natural to ask if there are general conditions 
that guarantee that ideals are arc closed. 
Note that, by \cref{th:intrinsic}, it suffices to look at the case of zero ideals
of local $k$-algebras. We have already observed that
ideals are not typically $m$-jet closed if $m \in \N$, even in a Noetherian setting; 
see \cref{r:non-mjc}.
Our first result in this direction shows that, in general, the arc closure is a nontrivial operation. 

\begin{prop}
\label{th:counterexample}
There exists a local $k$-algebra $R$ whose zero ideal is not arc closed.
\end{prop}

\begin{proof}
Let $R = k[[x_1,x_2,\dots]]$ be the power series ring in infinite
countably many variables, and let $\fm$ be its maximal ideal. 
Then the ideal
\[
\fa = (x_1-x_i^i \mid i \ge 2) \subset R
\]
is not $m$-jet closed for any $m \in \N\cup\{\infty\}$. In fact, 
the element $x_1 \in R$, which is not in $\fa$, belong to
$\mjc\fa$ for all $m \in\N \cup\{\infty\}$.
Indeed, for $m \in \N$ we have 
\[
x_1 \in (\fa + \fm^{m+1}) \subset \mjc\fa
\]
where the inclusion follows by \cref{th:m-adic}, and hence
\[
x_1 \in \bigcap_{m\in \N}\mjc\fa = \ac\fa
\]
where the equality follows by \cref{th:cap-mjc=ac}. This shows that $\fa \ne \ac\fa$.
We conclude by \cref{th:intrinsic} that the zero ideal of $R/\fa$ is not arc closed. 
\end{proof}

\begin{cor}
There exist germs $(X,o)$ of $k$-schemes that do not have
the embedded local isomorphism property.
\end{cor}

\begin{proof}
An example is given by $X = \Spec R/\fa$ where $R$ and $\fa$ are as in 
the proof of \cref{th:counterexample}.
The fact that $(X,o)$ does not have the embedded local isomorphism property
follows by \cref{th:emb-loc-iso-revised}. 
More explicitly, setting $Y = \Spec R/(\fa+(x_1))$, the inclusion 
of germs $\f \colon (Y,o) \to (X,o)$ gives a counterexample to the embedded local isomorphism property.
\end{proof}

The example given in the proof of \cref{th:counterexample} is not Noetherian. 
We do not know whether in the Noetherian setting every ideal is arc closed.

\begin{problem}
\label{q:Noetherian-case}
If $R$ is a Noetherian local $k$-algebra, is every ideal $\fa$ of $R$ arc closed?
\end{problem}

In the positive direction, we have the following results.

\begin{defi}
We say that local $k$-algebra $(R,\fm)$ is \emph{graded} if $R = \bigoplus_{n\ge 0} R_n$ is a graded 
algebra with maximal ideal $\fm = \bigoplus_{n\ge 1} R_n$. We say that a germ $(X,o)$ of a $k$-scheme
is \emph{homogeneous} if $\O_{X,o}$ is a graded local $k$-algebra.
\end{defi}

\begin{thm}
\label{th:ac-trivial}
The zero ideal of $(R,\fm)$ is arc closed in the following cases:
\begin{enumerate}
\item
$R$ is a graded local $k$-algebra;
\item
$R$ is a reduced Noetherian local algebra essentially of finite type over $k$; 
\item
$R = S/(f)$ where $S$ is a regular ring essentially of finite type over $k$ and
$f \in S$.
\end{enumerate}
\end{thm} 

\begin{proof}
We first prove~(a).
By \cref{th:mjc=ker}, we reduce to check that the homomorphism
$\l_\infty^* \colon R \to (R_\infty/\fm R_\infty)[[t]]$ is injective.
Consider the $R$-valued arc $\a$ defined by the homomorphism
\[
\a^* \colon R \to R[[t]]
\]
given on homogeneous elements $g_n \in R_n$ by $\a^*(g_n) = g_nt^n$. 
Since $\a^*(\fm) \subset (t)$, $\a^*$ factors through $\l_\infty^*$
by the universal property. 
Since $\a^*$ is injective, it follows that $\l_\infty^*$ is injective too. 

Regarding~(b) and~(c), 
we observe that in both cases we can write $R \simeq S/\fb$ where 
$S$ is a regular ring essentially of finite type over $k$ and $\fb$ is an integrally
closed ideal of $S$. In this case we know that $\fb$ is jet support closed
by \cref{jsc<integral-closure}, and hence arc closed by \cref{th:ac<jcs}. 
Then we conclude by \cref{th:intrinsic} that the zero ideal of $R$ is arc closed. 
\end{proof}

\begin{cor}
\label{th:emb-property-holds}
The following germs of $k$-schemes have the embedded local isomorphism property:
\begin{enumerate}
\item
homogeneous germs; 
\item
germs of reduced schemes of finite type over $k$; 
\item
germs of hypersurfaces in smooth $k$-varieties. 
\end{enumerate}
\end{cor}

\begin{proof}
By \cref{th:emb-loc-iso-revised,th:ac-trivial}.
\end{proof}

In the geometric setting, \cref{q:Noetherian-case} leads to the following question.

\begin{problem}
Do germs of Noetherian $k$-schemes have the embedded local isomorphism property?
\end{problem}

This is one of those situations in which either a positive or negative solution
would be a positive outcome. 
If the problem is affirmatively solved, then it would imply that, in the Noetherian setting, 
the local jet-schemes determine completely the germ of the singularity, which is an interesting property
of jet schemes.
On the other hand, if the problem is negatively solved, then it would mean that
even in the setting of ideals of regular Noetherian rings, 
there exists a geometrically meaningful non-trivial closure operation
that is tighter than integral closure.

\end{document}